\documentclass{amsart}
\usepackage[utf8]{inputenc}
\usepackage{amsmath, amssymb, amsthm}
\usepackage{amscd,tikz,tikz-cd}
\usepackage{graphicx, enumitem}
\usepackage[margin=1in]{geometry}

\usepackage[mathscr]{euscript}
\usepackage{graphicx}
\usepackage{multicol} 

\usepackage{tikz}
\usepackage{tikz-cd}
\usepackage{cite}
\usepackage{mathtools}

\usepackage{verbatim}

\newcommand{\R}{\mathbb{R}}

\newcommand{\lists}[1]{\begin{enumerate} #1 \end{enumerate}}

\theoremstyle{plain}
\newtheorem{thm}{Theorem}[section]
\newtheorem*{thm*}{Theorem}

\newtheorem*{lemma*}{Lemma}

\newtheorem*{prop*}{Proposition}

\newtheorem*{cor*}{Corollary}

\theoremstyle{definition}
\newtheorem{defn}{Definition}[section]
\newtheorem*{defn*}{Definition}

\theoremstyle{remark}

\title{An infinite $\{3,7\}$-surface}
\author{Dami Lee}
\email{damilee@uw.edu} 
\address{University of Washington, Seattle, WA 98195} 
\date{}

\begin{document}

\maketitle

\begin{abstract}
A classical question in geometry is whether surfaces with given geometric features can be realized as embedded surfaces in Euclidean space. In this paper, we construct an immersed, but not embedded, infinite $\{3,7\}$-surface in $\R^3$ that is a cover of Klein quartic.\end{abstract}

\tableofcontents

\section{Introduction} 

A classical question in geometry is whether surfaces with given geometric features can be realized as embedded surfaces in Euclidean space. In this paper, our goal is to construct a polyhedral surface that observes the eight-step geodesic (the Petrie polygon that connects the midpoints of the edges of polygons) on the $\{3,7\}$-tiling of Klein quartic described in \cite{weber1998klein} and \cite{schulte1985}. The symbols $\{p,q\}$ are called Schl{\"a}fli symbols that describe a surface tiled by $p$-sided regular polygons with $q$ faces incident to each vertex.

In \cite{schulte1985}, the authors answer the question whether the $\{3,7\}$-tiling admits a geometric realization as a \emph{finite} polyhedron in Euclidean space. Another attempt of a geometric realization of Klein quartic appears on the sculpture ``The Eightfold Way'' at MSRI, as a $\{7,3\}$-tiling by 24 heptagons. Although the gluing pattern of the heptagons is regular on the sculpture, the heptagons themselves are not.

Inspired by the theory of triply periodic minimal surfaces, we are interested in the realization of the $\{3,7\}$-tiling as a quotient of a periodic (hence infinite) surface embedded in $\R^3.$ To identify the underlying structure of a surface, the more symmetry we have, the easier it is. Related research has been done in \cite{coxeter1937regular} where the authors find triply periodic polyhedral surfaces ($\{4,6\},$ $\{6,4\},$ and $\{6,6\}$) that observe large symmetry groups. In \cite{lee2017} and \cite{lee2018}, Lee expanded this result and showed that certain known surfaces can be obtained as quotients of triply periodic polyhedral surfaces. Here, quotients are taken by the periodicity, i.e., the lattice of translations. A triply periodic $\{3,8\}$-tiling has a compact quotient surface that yields the genus three Fermat quartic \cite{lee2017}, a $\{3,12\}$-tiling Schoen's minimal I-WP surface \cite{lee2018}, and a $\{4,5\}$-tiling the genus four Bring's curve \cite{lee2018}. 

For the $\{3,7\}$-tiling, it is unknown if a construction of a periodic tiling is feasible or not. In this paper, we show that there is an infinite surface immersed in $\R^3$ that is a cover of the $\{3,7\}$-tiling.

\begin{thm*} There exists an infinite $\{3,7\}$-surface immersed in $\R^3$ that is a cover of Klein quartic.\end{thm*}

\begin{figure}[htbp] 
	\centering
	\includegraphics[width=6in]{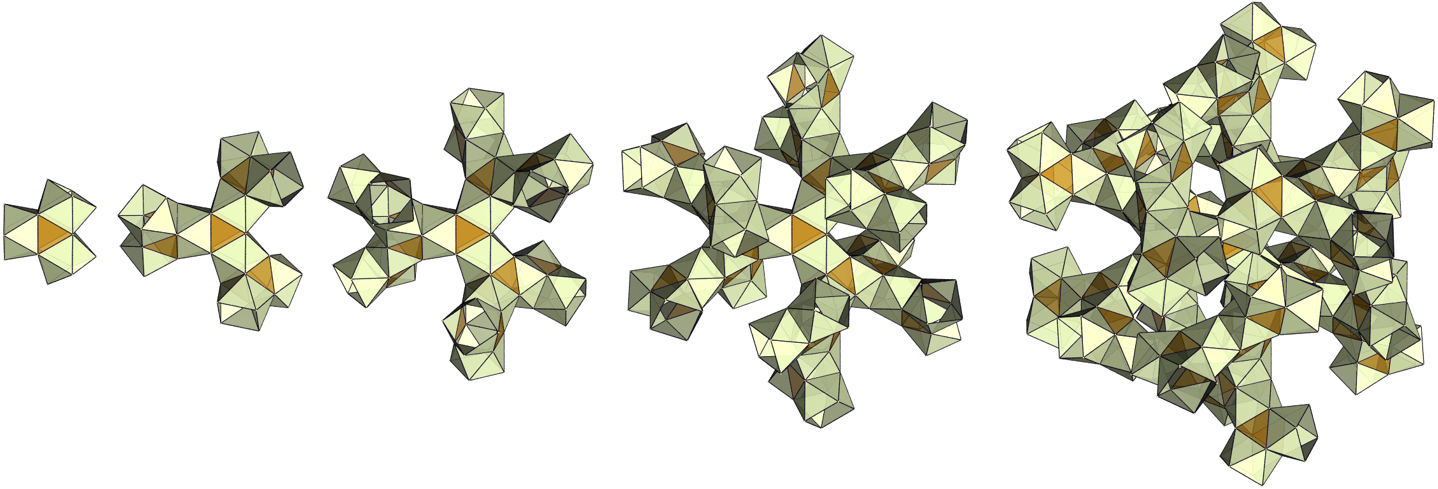}
	\caption{The evolution of an infinite $\{3,7\}$-surface.}
	\label{fig: evolution}
\end{figure}

We will explain the construction of the infinite surface in Section~\ref{sec: construction}. For now, we claim that after one more iteration from Figure~\ref{fig: evolution}, the surface will self-intersect. Hence this construction yields a surface that can only be immersed, but not embedded in $\R^3.$ We need to define what we mean by an infinite surface covering a compact surface, i.e. what we mean by the covering map. If the surface were embedded in $\R^3$ and furthermore periodic under translations, one can consider the compact quotient of the infinite surface by the lattice of translations.

We will show that the infinite $\{3,7\}$-surface is a covering of Klein quartic. We take a piece of the polyhedral surface tiled by 56 triangles and identify the edges so that a Petrie polygon in any direction is closed after eight steps (Figure~\ref{fig: funda}). This eight-step geodesic is identical to the one on Klein quartic \cite{weber1998klein}. The eight-step geodesic is the Petrie polygon that connects the midpoints of the edges of the $\left(\frac{2\pi}{7},\frac{2\pi}{7},\frac{2\pi}{7}\right)$-triangles. Alternatively, it can be seen as the strip of eight triangles in Figure~\ref{fig: 237 tiling}.

\begin{figure}[htbp] 
	\centering
	\includegraphics[width=2.5in]{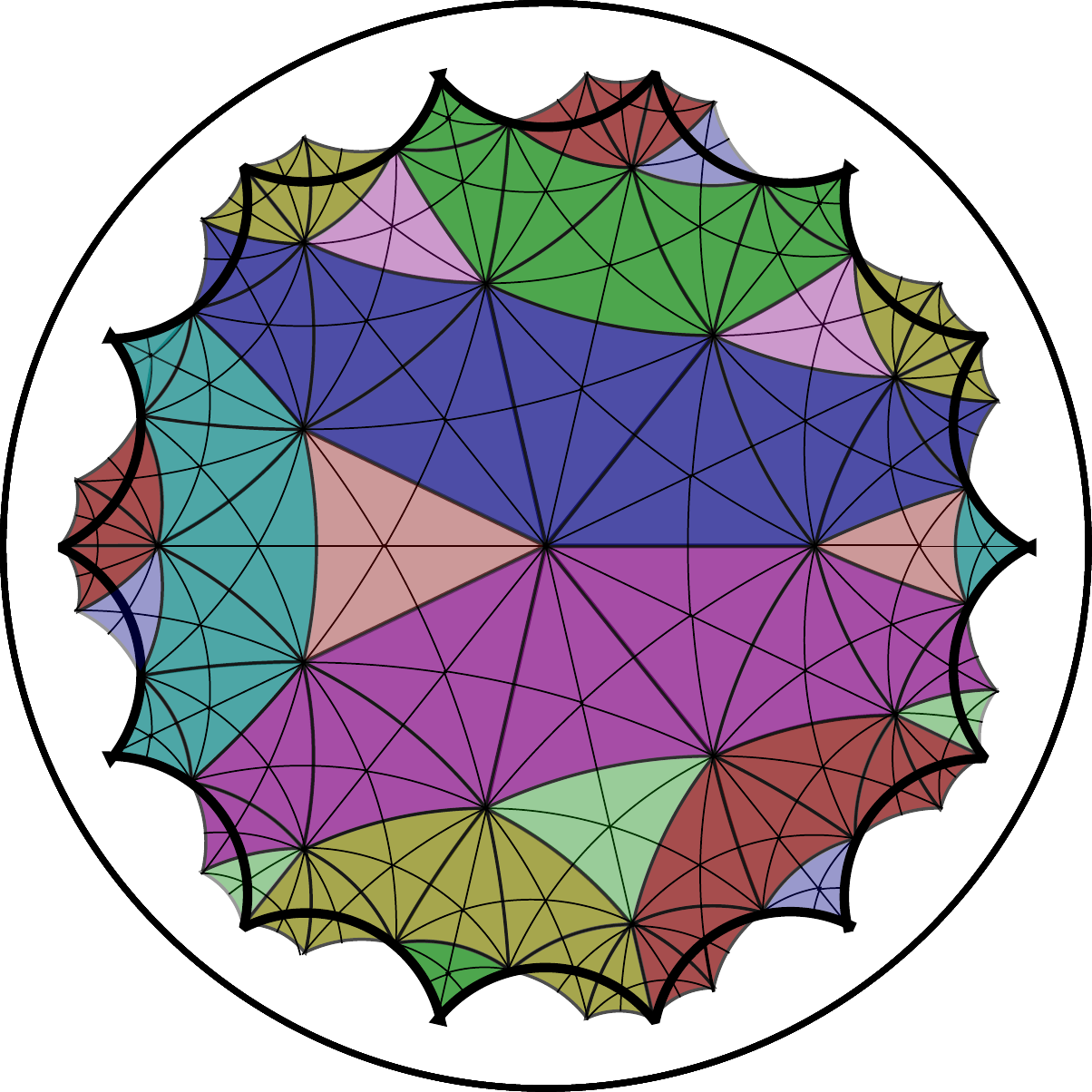}
	\caption{$(2,3,7)$-tiling on Klein's quartic}
	\label{fig: 237 tiling}
\end{figure}

This paper does not prove nor disprove the existence of an embedded or periodic $\{3,7\}$-surface whose compact quotient (via the lattice of translations) has the same Riemann surface structure as Klein quartic.

The author would like to thank Charles Camacho for the encouragement to publish this project, and the anonymous referees for their detailed reports that helped significantly improving this manuscript.

\section{Definition}

The idea of this construction is inspired by the notion of a \emph{decoration} of an embedded graph (Chapter 4 of the author's thesis \cite{lee2018}).

\begin{defn} Given a graph $\Gamma$ embedded in $\R^3,$ a \emph{decoration} of $\Gamma$ is a polyhedron achieved by replacing the 0-simplices and 1-simplices with convex polyhedral solids (including the empty solid) so that \lists{
\item there is a deformation retract of the polyhedron to the graph, and \item the solids are identified only along faces.}
In essence, if a 0-simplex and a 1-simplex in $\Gamma$ are incident, then their corresponding replacement solids in the decoration are identified along a face.\end{defn}

\section{Main result}\label{sec: construction} 

\begin{thm}\label{thm: construction} There exists an infinite $\{3,7\}$-surface immersed in $\R^3$ that is a cover of Klein's quartic.\end{thm}

\begin{proof} We build the surface with two types of Archimedean solids: a prism over an equilateral triangle and an anti-prism over a square. We let the quadrilateral sides of the triangular prism be squares, and the triangular sides of the anti-prism be equilateral triangles. 

We decorate a trivalent graph with triangular prisms replacing the 0-simplices, and square anti-prisms replacing the 1-simplices. We begin with a triangular prism and glue the base squares of the anti-prisms to the side squares of the prism. Hence, the square faces do not appear on the surface, and there are seven incident triangles at each vertex of the decoration (Figure~\ref{fig: evolution}).

Next, note that the base of the anti-prisms are squares, hence there are two different pairs of edges that can be identified to the edges that correspond to the ``height'' of the triangular prism (Figure~\ref{fig: chirality}). In other words, we have to choose in which direction we should ``twist'' the surface at each step of gluing the solids. We claim that this choice has to be made only once and the rest depends on the choice made at the first step, and that a twist in either direction (at the first step) yields a surface that observes the eight-step geodesic. Consider the horizontal reflection on Figure~\ref{fig: 237 tiling}. Iterations following the twist on the right of Figure~\ref{fig: chirality}, we obtain Figure~\ref{fig: evolution}.

\begin{figure}[htbp] 
	\centering
	\includegraphics[width=4in]{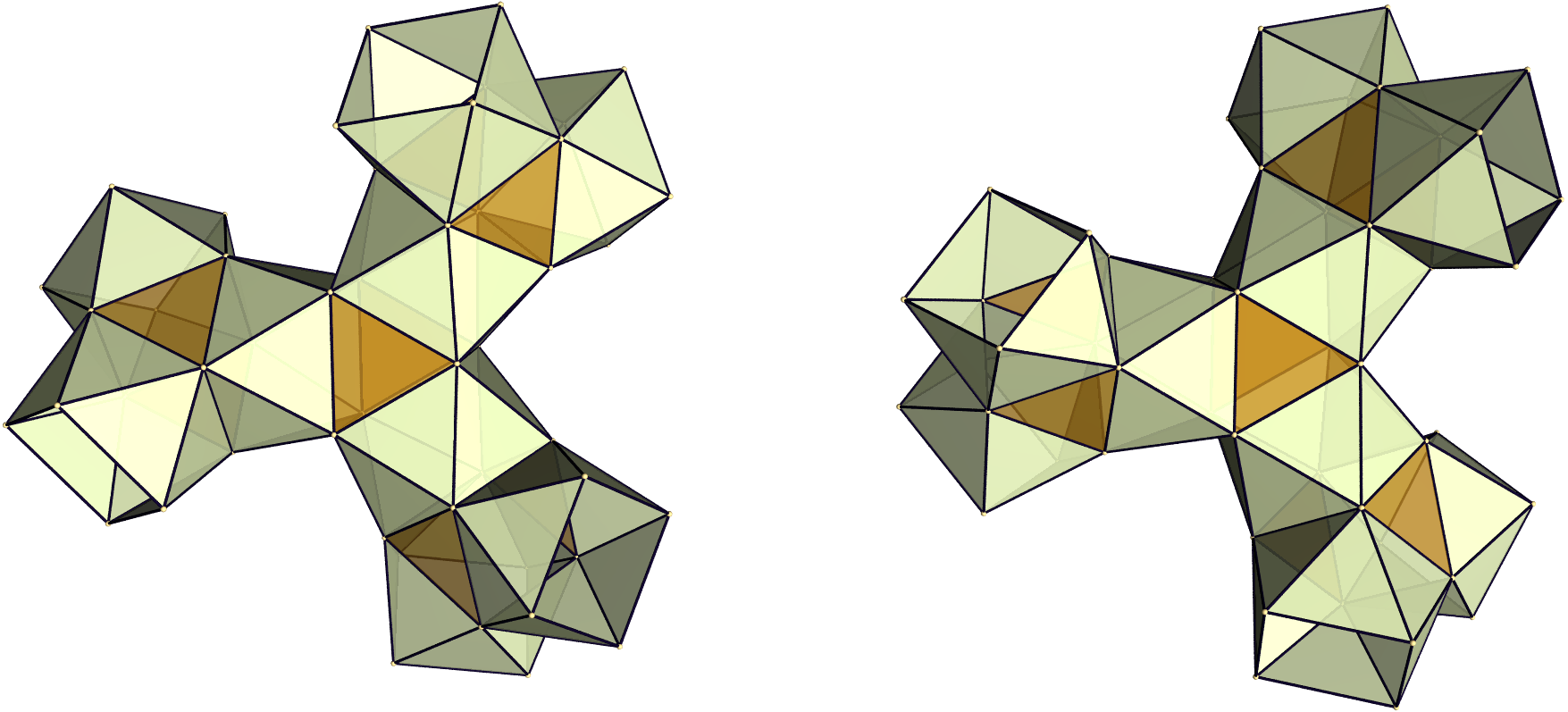}
	\caption{Twists in opposite directions.}
	\label{fig: chirality}
\end{figure}

Now we show that the surface is not embedded in $\R^3.$ One can check that by identifying the blue squares in Figure~\ref{fig: funda} ($E$ with $F,$ $E'$ with $F',$ and $E''$ with $F''$), one obtains a genus three surface that exhibits the eight-step geodesic identical to the one on Klein quartic. We will show by computing their normal vectors that these squares cannot be identified by Euclidean translation.

\begin{figure}[htbp] 
	\centering
	\includegraphics[width=2.5in]{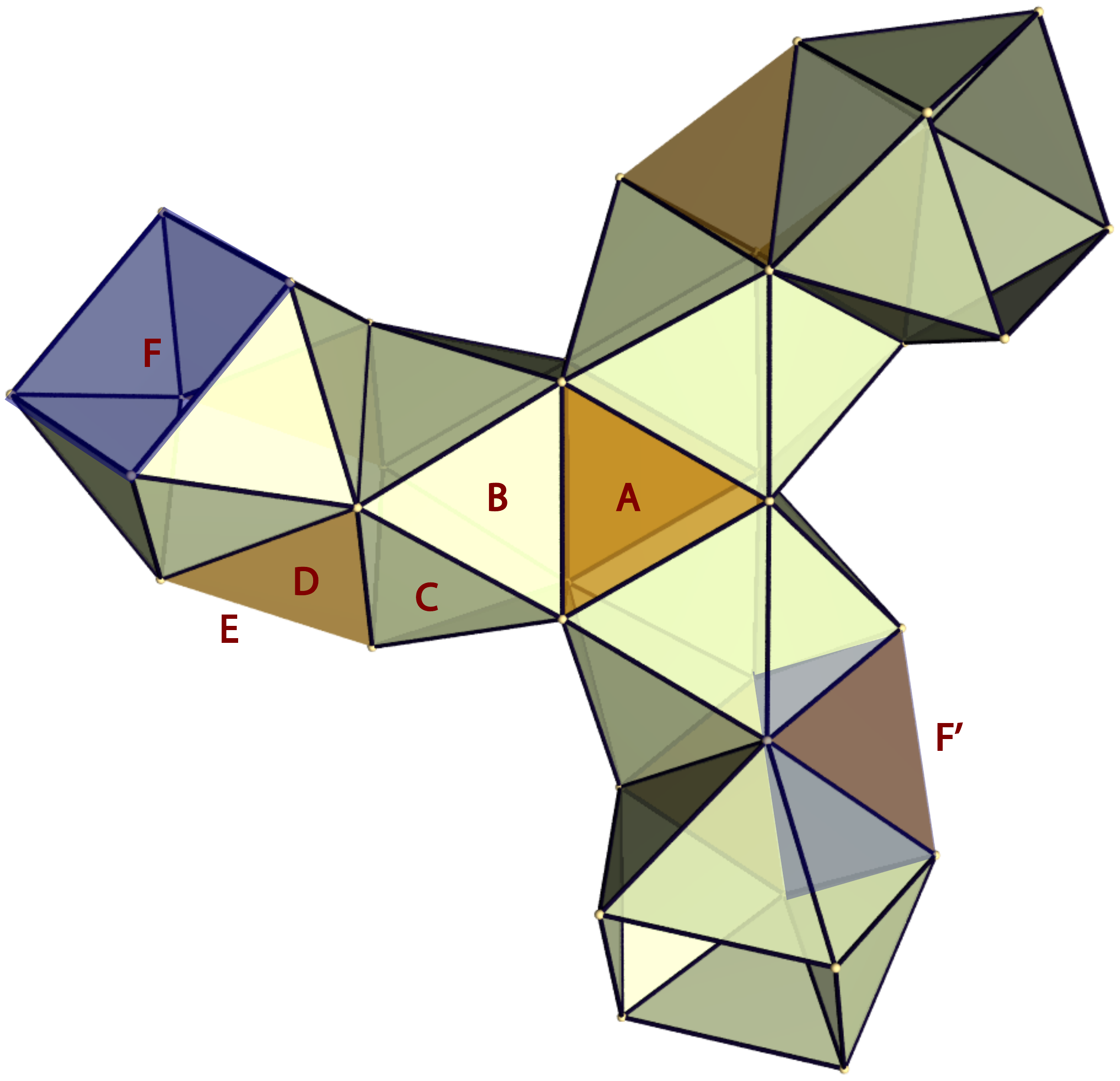}
	\caption{The quotient of the infinite $\{3,7\}$-surface.}
	\label{fig: funda}
\end{figure}

Here, in Figure~\ref{fig: funda}, we view the $x$-axis as pointing in the right direction, $y$-axis in the top direction, and $z$-axis pointing towards the reader. 
We locate the anti-prism including the triangle labeled $B$ so that its vertices lie at $(-h,\pm1,0),$ $(-h,0,\pm1),$ $\left(h,\frac{\sqrt{2}}{2},\pm\frac{\sqrt{2}}{2}\right)$ and $\left(h,-\frac{\sqrt{2}}{2},\pm\frac{\sqrt{2}}{2}\right).$ To ensure that the sides of the anti-prism are regular triangles, the value of $h$ must be $\frac{1}{\sqrt[4]{8}}\approx 0.59460356,$ 
hence all edges that appear on the polyhedral surface have length $\sqrt{2}.$ The triangular prism including the triangle labeled $A$ shares four of its vertices $\left(-h,\frac{\sqrt{2}}{2},\pm\frac{\sqrt{2}}{2}\right)$ with the anti-prism, and has two other vertices at $\left(h+\sqrt{\frac{3}{2}},0,\pm\frac{\sqrt{2}}{2}\right).$ It follows that the normal vector to triangle $A$ is $(0,0,1),$ and the (non-unit) normal vector to triangle $B$ is $\left(\sqrt{2}-1,0,2 \sqrt{2} h \right).$ 

We calculate the normal vectors to the labeled triangles and squares. Triangle $C$ has its vertices at $(h,-\frac{\sqrt{2}}{2},\frac{\sqrt{2}}{2}),$ $(-h,0,1),$ and $(-h,-1,0),$ and its (non-unit) normal vector is $(1-\sqrt{2},-2 h,2 h);$ triangle $D$ has its vertices at $(-h,-1,0),$ $(-h,0,1),$ and $(d,-0.5,0.5),$ where $d\approx -1.81934843$ 
is the solution to 

\begin{align*} (x+h)^2 + y^2 + (z-1)^2 = 2, \\
(x+h)^2 + (y+1)^2 + z^2 = 2,\\
(x-h)^2 + \left(y+\frac{1}{\sqrt{2}}\right)^2 + \left(z-\frac{1}{\sqrt{2}}\right)^2 = \left(2 h+\sqrt{\frac{3}{2}}\right)^2+ \left(1-\frac{1}{\sqrt{2}}\right)^2,\\
y=-\frac{1}{2}, z =\frac{1}{2}.
\end{align*}  

The normal vector to $D$ is $\left(0,-\frac{1}{\sqrt{2}},\frac{1}{\sqrt{2}}\right);$ square $E$ has its vertices at $(-h,0,-1),$ $(-h,-1,0),$ $(-d,-0.5,0.5),$ and $(-d,0.5,-0.5),$ and its (non-unit) normal vector is $(\sqrt{2},\sqrt{3},\sqrt{3}).$ Lastly, the normal vector to square $F'$ is $\left(6,\sqrt{6}+4\sqrt{3},\sqrt{6}-4\sqrt{3}\right).$ We can rotate the normal of $E$ and $F'$ by 120 or 240 degrees to achieve the normal vectors to $E', E'', F,$ and $F''.$ We rotate the normal of $F'$ by 240 degrees about the $z$-axis and the normal of $F,$  $(-6\sqrt{3}-\sqrt{6},2-\sqrt{2},-48+12\sqrt{2}).$ The normal vectors to squares $E$ and $F$ are not parallel to each other, hence they cannot be identified by translation in $\R^3.$ By the order-three rotational symmetry about the center of triangle $A,$ the other boundaries of the quotient ($E'$ and $F',$ $E''$ and $F''$) can be identified similarly.
\end{proof}


\addcontentsline{toc}{section}{Bibliography}
\begin{thebibliography}{10}
\bibliographystyle{abbrv}
\bibitem{coxeter1937regular} H. Coxeter. Regular skew polyhedra in three and four dimensions, and their topological analogues. \textit{Proceedings of the London Mathematical Society,} 43(2):33--62, 1937.

\bibitem{weber1998klein} H. Karcher and M. Weber. On Klein's Riemann surface. In \textit{The Eightfold Way,} volume 35, pages 9--49. MSRI Publications, 1998.

\bibitem{lee2017} D. Lee. On a triply periodic polyhedral surface whose vertices are Weierstrass points. \textit{Arnold Mathematical Journal,} 3(3):319--331, 2017.

\bibitem{lee2018} D. Lee, \textit{Geometric realizations of cyclically branched coverings over punctured spheres,} PhD thesis, Indiana University, 2018. 

\bibitem{schulte1985} E. Schulte and J. M. Wills, A Polyhedral Realization of Felix Klein's Map $\{3, 7\}_8$ on a Riemann Surface of Genus 3. \textit{Journal of the London Mathematical Society,} Volume s2-32, Issue 3, pages 539--547, 1985.
\end{thebibliography}
\end{document}